\documentclass[12pt]{article}

\usepackage{fullpage}

\usepackage[normalem]{ulem}

\usepackage[pdftex]{graphicx}
\usepackage{amsfonts}
\usepackage{amscd}
\usepackage{indentfirst}
\usepackage{amssymb,amsthm, amsgen, amstext, amsbsy, amsopn,yfonts}
\usepackage{xcolor}
\usepackage{amsmath}
\usepackage{array}
\usepackage{mathrsfs}
\usepackage[normalem]{ulem} %

\newtheorem{thm}{Theorem}[section]

\newtheorem{defin}[thm]{Definition}
\newtheorem{rem}[thm]{Remark}
\newtheorem{exam}[thm]{Example}

\newtheorem{assum}[thm]{Assumption}

\newcommand{\J}{{\mathbb{J}}}

\newcommand{\R}{{\mathbb{R}}}

\newcommand{\Z}{{\mathbb{Z}}}

\newcommand{\cP}{{\mathcal{P}}}

\newcommand{\bal}{\begin{aligned}}
\newcommand{\enbal}{\end{aligned}}
\newcommand{\be}{\begin{equation}}
\newcommand{\ee}{\end{equation}}
\newcommand{\farc}{\frac}
\newcommand{\pdr}[2]{\frac{\partial{#1}}{\partial{#2}}}

\newcommand{\Rm}{{\mathbb R}}

\usepackage[colorlinks=true,pdfpagemode=UseNone,urlcolor=blue,linkcolor=blue,citecolor=blue,breaklinks=true]{hyperref}

\begin{document}

\title{A contact topological glossary for non-equilibrium thermodynamics}
\author{Michael Entov$^1$,  Leonid Polterovich$^2$ and Lenya Ryzhik$^3$}

\footnotetext[1]{Technion -- Israel Institute of
Technology, Israel}
\footnotetext[2]{Tel Aviv University, Israel}
\footnotetext[3]{Stanford University, USA.}

\maketitle

\begin{abstract}
We discuss the occurrence of some notions and results from contact topology
in the non-equilibrium thermodynamics. This includes the Reeb chords and the partial order on the space of Legendrian submanifolds.
\end{abstract}

\section{Introduction}

The contact geometric formulation of equilibrium thermodynamics is classical and traces back to Gibbs (see \cite{Her} for a modern exposition).
In that approach, one starts with a manifold~$X$ formed by the intensive thermodynamic variables (prescribed external physical parameters).
Each  macroscopic thermodynamic state  is represented by a point in the jet space $J^1X= T^*X \times \R$     that includes the intensive variables,
the extensive thermodynamic variables (entropy and the generalized pressures) and the free energy. The starting point of the equilibrium
contact thermodynamics is the observation that the set of all thermodynamic equilibria forms a Legendrian
submanifold in $J^1X$.

In contrast, there exists a variety of contact geometric approaches to mathematical modeling of non-equilibrium thermodynamic
processes \cite{BLN,EP,EPR,Goto-JMP2015,Goto,Goto24,GLP,GrO,Grmela,Haslach,LimOh,V}. In this note, we present yet another approach that examines
slow (quasi-static) irreversible processes from the perspective of the partial order on the space of Legendrian submanifolds,
a profound geometric structure studied within contact topology \cite{AA,CN,CFP}; see also \cite{ElP}. The starting observation is that
the first and second laws of thermodynamics impose   a constraint: any time-dependent path generated by an  out-of-the-equilibrium
thermodynamic process
must be non-negative with respect to the standard contact form on $J^1X$ (see
Definition \ref{def-positive-curve} below), which is  a reformulation of the result in~\cite{Haslach}. The main goal of this note is to discuss the
consequences of this observation.

We consider three types of non-equilibrium processes that respect
the non-negativity of the thermodynamic paths:\\
(1) In the slow
(quasi-static)
global\footnote{The term ``global" here emphasizes that we consider processes defined at all macroscopic equilibrium states of the system, and not only at some of them.}
processes, the change in the  external   parameters is so slow that, at each moment in time, the system is in equilibrium.
Such processes are given by non-negative families (or paths) of Legendrian equilibrium submanifolds. \\
(2) In the fast processes, after an abrupt change in the extensive variables,
the system
converges, as~$t\to+\infty$, to the
equilibrium corresponding to the new values of the parameters but is not at an equilibrium at the intermediate times $t>0$.
Such processes are described by a thermodynamic path  driven by a Fokker-Planck equation for the
probability density on the microscopic thermodynamic states.  \\
(3)  In the ultrafast (instantaneous) processes, the system jumps instantaneously from
one equilibrium into another after an abrupt change in the external parameters. They correspond to Reeb chords
connecting the initial and the terminal equilibrium Legendrian submanifolds.

One natural question that can be addressed by the tools of contact topology is the existence of an ultrafast processes
that would connect the equilibria of two thermodynamic systems. As we discuss in Theorem~\ref{thm-B} below,
if two equilibrium thermodynamic Legendrians are related by a slow temperature non-decreasing global process,
there is a  free-energy non-increasing  ultrafast process which starts at an equilibrium state of the first system and ends at an equilibrium state of the second system.

\begin{rem}
\label{rem-temperature-disclaimer}
{\rm
It should be mentioned that when the process is not quasi-static (that is,
if we cannot assume that at each moment of time
the system is in the equilibrium), the temperature of the system is not well defined
from the physical point of view. We always
assume that the system exchanges heat with a thermostat (or reservoir) which has a well defined temperature, and for
a system not in an equilibrium, we will consider the
temperature of the reservoir at the given time. In an equilibrium, the temperature of the system is equal to the
temperature of the thermostat (see \cite{Str}, formula (44)). In physical terms, we tacitly
assume that on the ``boundary of the system"  the temperature equals to that of the thermostat, while ``inside" the temperature is defined only in an equilibrium.
In this paper, we mostly consider the temperature as an external
parameter and ignore its relaxation to the equilibrium. Elimination of the entropy and the temperature
from the thermodynamic phase space is called {\it reduction}, see Sections \ref{sec:reduced} and \ref{sec:reduced-1}.

}
\end{rem}

We finish this brief introduction with Table~\ref{table-1} that sums up a contact topological glossary for thermodynamics.
\begin{table}[h]
\centering
\begin{tabular}{|l|l|}
\hline
Contact manifold & (Reduced) thermodynamic phase space \\ \hline
Contact form & The Gibbs fundamental form \\ \hline
Legendrian submanifold & (Reduced) equilibrium submanifold \\ \hline
A non-negative path &
 A non-equilibrium (temperature non-decreasing)  \\
 &  \\ \hline
A non-negative path of Legendrians & Slow (temperature non-decreasing) global process \\ \hline
Reeb chord between two Legendrians & An instant relaxation in the thermodynamic limit \\ \hline
\end{tabular}
\caption{Thermodynamic glossary}\label{table-1}
\end{table}

{\bf Organization of the paper.} Section~\ref{sec:glossary} discusses the construction of the equilibrium Legendrians starting
with the microscopic models and passing to the thermodynamic phase space. We also introduce the reduction process that allows
us to formally freeze certain intensive and extensive thermodynamic variables.  In Section~\ref{sec:positive} we explain the
relation between the first and second laws of thermodynamics and the positivity of the thermodynamic paths in phase space.
The definitions of the slow, fast and ultrafast processes are discussed in Section~\ref{sec:classify}.  Section~\ref{sec:slow} uses some non-trivial
results in symplectic topology to obtain
certain consequences of the existence of slow
global
processes connecting the equilibrium Legendrian submanifolds for two given thermodynamic
systems. Finally, in Section~\ref{sec-ultra}, we discuss the ultrafast processes in the examples of the ideal gas and Stirling engine, and
the Curie-Weiss magnet. Interestingly, for the Stirling engine, these ultrafast processes seem to be also observed experimentally~\cite{Stirling-Nature}.

\medskip
\noindent
{\bf Acknowledgment.} We thank Bernard Derrida, Leonid Levitov, Shin-itiro Goto, and Nikita Nekrasov for very useful comments.
Preliminary results of this paper were presented by LP at the Symplectic Geometry Seminar in
IAS, Princeton. The recording of this talk, called ``Contact topology meets thermodynamics", is available on YouTube. The paper was written during LP's sabbatical at the University of Chicago. LP thanks the Department of Mathematics for an exceptional research atmosphere.
ME was partially supported by the Israel Science Foundation grant 2033/23,
LP was partially supported by the Israel Science Foundation grant 1102/20, and
LR was partially supported by NSF grant  DMS-2205497 and by ONR grant N00014-22-1-2174.

%

\section{Thermodynamic equilibria as the Legendrian sub-manifolds}\label{sec:glossary}

The contact geometric notions become relevant for thermodynamics when it
is considered in the thermodynamic phase space -- an odd
dimensional space with the coordinates given by a thermodynamic potential and pairs
of conjugate thermodynamic variables. It is equipped with the Gibbs fundamental form,
which endows it with a contact structure. We shall deal with two versions of the thermodynamic
phase  space, the extended one described in Section~\ref{sec:extended}, and the reduced one introduced in
Section~\ref{sec:reduced}.

For reader's convenience, we first recall in Section~\ref{sec:micro} the connection between the original microscopic description
of thermodynamic equilibria and the macroscopic picture that eventually leads to the thermodynamic phase space.
We follow the framework  of~\cite{EPR}, with some minor modifications.

\subsection{Microscopic and macroscopic thermodynamic states}\label{sec:micro}

A microscopic state of a thermodynamic system is a point on a manifold~$M$,  equipped
with a measure  $\mu$. We let $\mathcal P$ be the collection of smooth positive densities  $\rho>0$
on $M$
so that~$\int_M \rho d\mu = 1$, and refer to an element  $\rho\in\mathcal P$ as a macroscopic  state.
Its entropy  is
\be\label{25feb1104}
S(\rho) = -\int_M \rho \ln \rho~d\mu,~~ {\rho\in{\mathcal P}}.
\ee

The free energy of the system is
\begin{equation}\label{eq-Phi-rho}
G(T,q,\rho)=
-T S(\rho) + \int_M H(q,m) \rho(m)d\mu(m),~~q\in\Rm^n,~\rho\in\cP.
\end{equation}
Here, $q\in\Rm^n$ is an exterior physical parameter  and
$T>0$ is the  temperature of the system
(see Remark~\ref{rem-temperature-disclaimer}).
The Hamiltonian $H(q,m)$ often appears in the form
\be\label{25feb604}
H(q,m)=V_{\rm int}(m)+V_{\rm ext}(q,m),
\ee
where $V_{\rm int}(m)$ is the
 microscopic
internal energy, and $V_{\rm ext}(q,m)$ is related to an external influence on the system, as indicated
by its dependence on the exterior parameter $q\in\Rm^n$.
A special case of interest is when $V(q,m)$ is linear in $q$:
\be\label{25feb602}
V_{\rm ext}(q,m)=(q\cdot \overline V(m)).
\ee
Here, $(\cdot)$ denotes the standard inner product on $\Rm^n$, and  $\overline V$ is a map
from $M$ to $\R^n$.

For every $T>0$,
and  $q\in\Rm^n$ fixed,
the functional $G(q,\cdot)$ is strictly convex on~$\mathcal{P}$.
A direct computation shows that its
unique critical point, known as the Gibbs distribution, or an equilibrium distribution, is
\begin{equation}\label{eq-var}
{\rho_G:=\hbox{argmin}_{\rho\in{\mathcal P}}G(T,q,\rho)}=\frac{e^{-\beta H(q,m)}}{\mathcal{Z}(q)},~~
\mathcal{Z}(q) = \int_M e^{-\beta H(q,m)} d\mu(m).
\end{equation}
%
%
%
%
%
Note that the entropy of a macroscopic state $\rho\in\cP$ can be written as
\be\label{25feb1002}
S(\rho)=-\pdr{G}{T}(T,q,\rho).
\ee
Similarly,
for given a macroscopic state~$\rho\in\cP$
the generalized pressures~\cite{Be} are defined as
\begin{equation}
\label{eq-free-en-gf}
p_j{(q,\rho)} = -\frac{\partial G}{\partial q_j}  {(T,q,\rho)}
= -\int_M \frac{\partial H}{\partial q_i} {(q,m)} \rho {(m)} d\mu({m})\;, \; j=1,\dots , n\;.
\end{equation}
Note that $p_i$ do not depend on $T$. The entropy and the generalized pressures are,
in the sense of (\ref{25feb1002})-(\ref{eq-free-en-gf}),
the dual variables to the external physical parameters $T>0$ and $q\in\Rm^n$, respectively.

In the special case of the linear  Hamiltonians as in (\ref{25feb604})-(\ref{25feb602}), the pressures have the form
\[
p_j(\rho)=-\int_M \bar V_j(m) \rho(m)d\mu(m),~~j=1,\dots,n,
\]
and the free energy in (\ref{eq-Phi-rho}) is
\begin{equation}\label{25feb608}
\bal
G(T,q,\rho)&=U(\rho)-TS(\rho)-\sum_{j=1}^n p_j(\rho)q_j.
~~q\in\Rm^n,~\rho\in\cP,
\enbal
\end{equation}
 where $U(\rho)$, called the macroscopic internal energy, is defined by
\be\label{25feb1004}
U(\rho)=\int_M V_{\rm int}(m) \rho(m)d\mu(m).
\ee

\subsection{A geometric interpretation of the equilibria and the extended thermodynamic phase space}
\label{sec:extended}

These basic notions can be represented geometrically as follows.
Given a macroscopic state $\rho\in\cP$, we consider
the point $(-G(T,q,\rho), S(\rho), T, p(q,\rho),q)$ as an element of $\Rm^{2n+3}$ with the coordinates
$z=-G$, $S,T$, and $(p_j,q_j)$, $j=1,\ldots,n$. This space
is called {\it the extended thermodynamic phase space} and is
denoted by $\widehat{\mathcal{T}}$.
We will refer to $T$ and $q=(q_1,\ldots,q_n)$ as the intensive variables, and to~$S$ and $p=(p_1,\ldots,p_n)$ as the
extensive variables.

The space $\widehat{\mathcal{T}}$ is naturally identified with the jet bundle
\[
J^1\R^{n+1}:= \R (z) \times T^*\R^{n+1}\;,
\]
equipped with the canonical $1$-form
\be\label{25feb612}
\widehat{\lambda} = dz-SdT - \sum_{j=1}^n p_j dq_j,
\ee
called the Gibbs 1-form. It is a contact form: the maximal dimension of
an integral submanifold  $\widehat\Lambda$ of the distribution $\text{ker}(\widehat{\lambda})$
equals $n+1$.
 Recall that  $\widehat\Lambda$ is called integral if  it is tangent to the distribution,
that is,  $\widehat{\lambda}$ vanishes on~$T\widehat\Lambda$.
The integral submanifolds of the maximal dimension $n+1$ are called {\it Legendrian}.

In our setting, a standard example of a Legendrian submanifold is  the
1-jet of a smooth function $f: \R^{n+1} \to \R$,

\begin{equation}\label{eq-lambdaef}
\widehat\Lambda_{f} = \Big\{ (z,S,T,p,q,z)\in{\J^1\R^{n+1}}\Big|~z= f(T,q),~S=\pdr{f}{T}(T,q),~
p_j = \frac{\partial f}{\partial q_j}(T,q), j=1,\ldots,n\Big\}\;.
\end{equation}
More generally, given a {smooth}
function $$\Psi: \R^{n+1} \times {\mathcal{A}} \to \R\;,$$
where $\mathcal{A}$ is a linear
space of auxiliary variables $\xi$,
the set $\Lambda_\Psi$ of all $(z,S,T,p,q)\in\J^1\R^{n+1}$ such that there
exists $\xi\in \mathcal{A}$ so that
\begin{equation}\label{eq-gf}
\bal
\widehat{\Lambda}_{\Psi} = &\Big\{\hbox{$(z,S,T,p,q):\exists\xi\in \mathcal{A}$ s.t }
z= \Psi(T,q,\xi),~\frac{\partial \Psi}{\partial \xi} (T,q,\xi) = 0,~\\
&S= \frac{\partial \Psi}{\partial T} (T,q,\xi),~
p= \frac{\partial \Psi}{\partial q} (T,q,\xi) \Big\},
\enbal
\end{equation}
is a (possibly singular) Legendrian submanifold for a generic $\Psi$. The function $\Psi$
is called a generating function of $\widehat{\Lambda}_{\Psi}$, and the variables $\xi$ are called
the ghost variables.

A key  observation is that (\ref{eq-var})-(\ref{eq-free-en-gf}) exactly mean that
a Gibbs measure $\rho_G(T,q,\cdot)$
of a thermodynamic system corresponds to a point on the Legendrian submanifold $\Lambda$
of the form~(\ref{eq-gf}) with the elements $\rho\in \cP$ being
the ghost variables $\xi$. The generating
function of~$\Lambda$ is the (minus) free energy $-\Phi$
(we learned this observation from \cite{LimOh}).
In other words, the equilibria of a thermodynamic
system form a Legendrian submanifold $\widehat\Lambda$ in $\widehat{\mathcal T}$, as we vary the external parameters
$T$ and $q$, that themselves are coordinates on $\widehat{\mathcal T}$.


%
%
%

\subsection{The reduced thermodynamic phase space} \label{sec:reduced}

The aforementioned Legendrian submanifold $\widehat\Lambda$
includes all possible
equilibria, for all values of $T>0$ and $q\in\Rm^n$. As the physical parameters change,
one can, in principle, think of time-dependent thermodynamic non-equilibrium processes
that start and end at two corresponding equilibria as paths connecting
two points on $\widehat\Lambda$. This, however, is somewhat unwieldy, and it is convenient
to work in  the {\it reduced} thermodynamic phase space.
With an eye toward  applications, we present a rather  general reduction procedure.
Fix~$1 < k \leq \ell \leq n$ and set
\[
I= \{k+1,\dots, \ell\},~~E = \{\ell+1,\dots, n\}.
\]
Consider the
affine subspace $h \subset J^1\R^{n+1}$  obtained
by fixing the values of the temperature and the intensive variables with indices from $I$, and
setting the extensive variables with indices from $E$ to be zero:
\be\label{25feb614}
h = \{T=T^0;  q_i = q_i^0, \; i \in I;  \; p_e = 0,\; e \in E\}\;.
\ee
This is an $n+2+k$ dimensional sub-space.
The  reduced thermodynamic phase
space $\mathcal{T}$ is obtained by the projection of $h$ along the directions
corresponding to the variables $S$, $p_i, i \in I$ and $q_e,  e \in E$,
and is $2k+1$ dimensional.

\begin{rem}{\rm
Note that our notions of extended and reduced thermodynamic phase spaces and of reduction differ from the notions with the same names appearing in some other papers on geometric approaches to thermodynamics - see e.g. \cite{vdSchaft}.}
\end{rem}

The reduced thermodynamic phase space is naturally identified with $J^1\R^k$.
Consider now the projection $J^1\R^{n+1} \to J^1\R^k$,
\[
(z,S,T,p_1,q_1,\dots,p_n,q_n) \mapsto (z,p_1,q_1,\dots,p_k,q_k)\;,
\]
and let
\be\label{25feb1102}
\lambda=dz-\sum_{j=1}^kp_jdq_j,
\ee
be the standard contact structure on $J^1\R^k$.
Given an equilibrium Legendrian submanifold~$\widehat{\Lambda} \subset \widehat{\mathcal{T}}$,
its reduction
$\Lambda$ is defined as the image of $\widehat{\Lambda} \cap h$ under the above-mentioned projection.
It follows from (\ref{25feb612}) and the definition (\ref{25feb614}) of $h$ that, under certain transversality assumptions,
the set $\Lambda$ is also a Legendrian submanifold of $J^1\R^k$ with
respect to the standard contact structure. Physically, $\Lambda$ represents the collection
of all equilibria of the system with the temperature $T$ and the
intensive variables $q_i$, $i\in I$ fixed and $p_e=0$, $e\in E$.

Enlarging the original system by adding unspecified thermodynamic variables (for example, considering a system of pair-wise interacting Curie-Weiss magnets
discussed in Example~\ref{exam-magnetic}  below)
and reducing them at some fixed values will lead to more complicated families of Legendrians which could be considered
as perturbations of $\Lambda$.
In fact, one can get  in this way every Legendrian which is
obtained from the zero section by a compactly supported Hamiltonian isotopy.
This fact and its extension to jet bundles of manifolds follows from the theory of generating functions, (see e.g. \cite{CFP} for a detailed discussion). It would be interesting to elaborate meaningful examples of such perturbations.

Passing to the next level of abstraction, we can assume that the space of intensive variables is a smooth manifold $X$ and
the thermodynamic phase space (either extended, or reduced) is $J^1X= \R(z) \times T^*X(p,q)$ equipped with the contact form $\lambda=dz-pdq$.
The submanifolds~$\Lambda$ consisting of all equilibria of such system
are Legendrian: $\dim \Lambda = \dim X$ and $\lambda|_{T\Lambda} = 0$.


\begin{rem}\label{rem-infty}
{\rm In the context of thermodynamics, there is a  discrepancy between mathematical
theory which requires a reasonable behavior ``at infinity" of the Legendrian submanifolds
and the pool of physical examples. This resembles the situation in classical mechanics where
the best theory exists for the compactly supported case, while physically meaningful Hamiltonians are not compactly supported.
Nevertheless, sometimes one can close this gap by topological tricks,
see e.g. the proof of Theorem \ref{thm-main} below for a reduction of a physical problem to $J^1S^1$.
}
\end{rem}

%
%

\section{Positivity of the thermodynamic paths}\label{sec:positive}

In this section, we will see that the laws of thermodynamics
imply that the paths  that correspond to physical
processes must be non-negative with respect to the corresponding contact forms both
in the extended and reduced thermodynamic phase spaces, a minor modification of a result
of~\cite{Haslach}.

\subsection{The thermodynamic paths in the extended thermodynamic phase space}

We will now consider, for the sake of convenience, thermodynamic systems with
affine Hamiltonians of the form~(\ref{25feb604})-(\ref{25feb602})
so that the free energy has the form (\ref{25feb608}).
So far, we have described the collections of thermodynamic equilibria as
points on the Legendrian submanifolds, either in the full thermodynamic phase space
or in its reduced version.  
A {\it thermodynamic path} is a smooth
 path
\be\label{25feb1010}
\gamma(t)=(-G(T(t),q(t),\rho(t)),S(\rho(t)),T(t),p(q(t),\rho(t)),q(t)) \subset \widehat{\mathcal{T}},
\ee
in the phase space
generated by a smooth time-dependent process, with $T(t)>0$, $q(t)\in\Rm^n$ and~$\rho(t)\in{\cP}$,
that may be out-of-the-equilibrium.
As we have already mentioned in Remark~\ref{rem-temperature-disclaimer},
we always
assume that the system exchanges heat with a thermostat (or reservoir) which has a well defined temperature, and  the coordinate~$T(t)$ of the path~$\gamma$ is the
temperature of the reservoir
at the time $t$.


A reformulation of an important observation of~\cite{Haslach} is that the laws of thermodynamics
impose a constraint on the paths in the thermodynamic phase that can be physically realized.
To this end, let us recall the following


\begin{defin}\label{def-positive-curve} {\rm A
 path
 $\gamma(t) \subset J^1\R^{n+1}$ is
called {\it non-negative} if $\widehat{\lambda}(\dot{\gamma}(t)) \geq 0$ for every $t$.
}
\end{defin}

Let us emphasize that the definition does not depend on the choice of the specific contact
form $\widehat{\lambda}$ but only on the contact distribution and its co-orientation.
%
%
%
To derive the aforementioned
constraint, let $(T(t),q(t)),\rho(t))$ be a smooth path and
$\dot\gamma(t)$ and $\dot\rho(t)$ be, respectively, the tangent vectors to the corresponding thermodynamic path $\gamma(t)$ in (\ref{25feb1010})
and to $\rho(t)\in\cP$.
By the first law of thermodynamics, the infinitesimal
increment of the
 macroscopic
internal energy defined in (\ref{25feb1004})
can be written as the difference
\begin{equation}\label{eq-1law}
dU(\dot\rho) = \Delta Q - \Delta W(\dot\gamma)\;.
\end{equation}
Here, $\Delta Q$ is the infinitesimal amount of heat supplied to the system, and
\be\label{25feb1012}
\Delta W(\dot\gamma) = -\sum_{j=1}^n q_j dp_j(\dot\gamma)
\ee
is the work done by the system (see \cite{Ja}).  In fact Jarzynski in \cite{Ja}
distinguishes between two types of work, inclusive and exclusive, and the one given in \eqref{25feb1012} is the exclusive one.
In addition,
by the second law of thermodynamics, the increment of the entropy is
\begin{equation}\label{eq-2law}
dS(\dot\gamma) = \frac{\Delta Q}{T} +d_{irr} S(\dot\gamma).
\end{equation}
Here, $d_{irr} S(\dot\gamma)$ is the {\it irreversible entropy production}, see \cite{De}, equation (3.8).
The second law of thermodynamics states that $d_{irr} S(\dot\gamma) \ge 0$ for all $t>0$.
If $d_{irr} S(\dot\gamma) = 0$  for all $t$, the process is reversible. Otherwise, it is called irreversible.
Combining~(\ref{eq-1law}),~(\ref{25feb1012}) and (\ref{eq-2law}), together with the expression (\ref{25feb608}) for
the free energy and the definition (\ref{25feb612}) of $\widehat\lambda$, 
we get
\[
Td_{irr} S(\dot\gamma) = -dU(\dot\rho) +TdS(\dot\gamma)+ \sum_{j=1}^n q_j dp_j (\dot\gamma) = \widehat{\lambda}(\dot\gamma)\;.
\]
We conclude from the second law of thermodynamics that 
\be\label{25jan2806}
\hbox{thermodynamic paths in $\widehat{\mathcal{T}}$  are given by non-negative paths $\gamma$},
\ee
an observation that can be found in a slightly different form in~\cite{Haslach}. Note that for a reversible process we always have
$\widehat{\lambda}(\dot\gamma)=0$.
This   is closely related to the fact that the set of equilibria of a system forms a Legendrian submanifold.

\subsection{The thermodynamic paths in the reduced thermodynamic phase space}\label{sec:reduced-1}

%
%
We now discuss the non-negativity of the thermodynamic paths in the reduced thermodynamic phase space.  Suppose that the reduction is made with respect to the intensive variables
with the indices from $I$ and the extensive variables with the indices from $E$, where
\[
I= \{k+1,\dots, \ell\},~~E = \{\ell+1,\dots, n\}.
\]

Consider   the projection $J^1\R^{n+1} \to J^1\R^k$,
\[
(z,S,T,p_1,q_1,\dots,p_n,q_n) \mapsto (z,p_1,q_1,\dots,p_k,q_k)\;,
\]
and let $\lambda$ be the standard contact structure on $J^1\R^k$, as defined in (\ref{25feb1102}).
Take any thermodynamic path
\[
\widehat{\gamma}(t) = (z(t),S(t),T(t),p(t),q(t))\in\widehat{\cal T}.
\]

\begin{assum} \label{assum-red} While considering  thermodynamic processes given by a path in the reduced thermodynamic phase space, we assume that
\begin{equation} \label{eq-Aincl}
A:= S(t)\dot T(t)+\sum_{j \in I} p_j(t)\dot q_j(t)\ge 0\;.
\end{equation}
Furthermore, we assume that $p_e=0$, $e \in E$.
\end{assum}

Let us illustrate inequality \eqref{eq-Aincl} in some situations.
If the temperature $T$ is the only reduced variable and the process is isothermal,
then $A=0$ and
Assumption~3.2 is automatically satisfied.

If $T$ is increasing in the course of the process, we shall assume that the entropy
is non-negative. Recall that, absent a specific normalization (such as that provided by the third law), entropy is defined only up to an additive constant, since only entropy differences have physical meaning. Accordingly, we may replace $S$ by a modified entropy $S+S_0$, where
$S_0$ is a constant chosen so as to ensure positivity. To keep the description consistent, we simultaneously modify the free energy by $G \mapsto G - S_0 T $. Geometrically, the replacement $S \mapsto S+S_0$, $G \mapsto G - S_0 T$, induces a contactomorphism of the extended thermodynamic phase space which preserves the canonical form \eqref{25feb612}. Thus, this change of normalization
does not affect the notion of positive thermodynamic paths or the validity of the second law inequalities discussed in this section.

For example, suppose that $\rho$ is a probability density on a finite set of $N$ points
and that $\mu$ is the uniform probability measure on this set. Then the probability of
point $i$ equals $p_i=\rho_i/N$, and hence
\[
S(\rho) = - \sum_i p_i \log p_i - \log N .
\]
It follows that choosing $S_0=\log N$ makes the modified entropy $S+S_0$ non-negative.

Back to the discussion on Assumption \ref{assum-red}, look at 
the ideal gas example
considered in Section~\ref{sec:ideal-gas} below. Here \color{black}
the volume $V$ and the negative pressure~$-P$ are pair-wise conjugate extensive and intensive variables, respectively, and the volume is, of course, always non-negative.
Thus, if we reduce the pressure only, the assumption above reads that the pressure is non-increasing.
On the other hand, in the Curie-Weiss model discussed in Example~\ref{exam-magnetic}
and Section~\ref{subsec-mag} below, the magnetization $M$ (dual to the external
magnetic field $H$ that is an intensive variable) is an extensive variable that
is not necessarily positive. In that context, when the external magnetic field is reduced,
assumption (\ref{eq-Aincl}) does not directly translate into its monotonicity in time.

A more general
interpretation of \eqref{eq-Aincl} comes from the fact that in an equilibrium it can be
written as
\begin{equation}
\label{eq-Aincl-1}
A = -\frac{\partial G}{\partial T} \dot{T} - \sum_{j \in I}\frac{\partial G}{\partial q_j}\dot{q_j}\;.
\end{equation}
Thus, the first term in the right side of (\ref{eq-Aincl})
is the rate of change of the free energy due to
the change in the temperature, while the second term
is, in the terminology of \cite{Ja}, the inclusive work done
by the system. In other words $A$ is {\it the decrement of the free energy due to the change of
the reduced intensive variables from} $I$. We shall tacitly assume that the same interpretation
holds for non-equilibrium processes which are close to equilibrium ones (e.g., for quasi-static processes considered in Section \ref{sec:classify} below).  For such processes
the first part of Assumption \ref{assum-red} can be restated as follows:
\be
\bal
&\hbox{Throughout the process, the reduced variables contribute either a decrease}
\\
&\hbox{or no change to the free energy.}
\enbal
\ee

Let
\[
\gamma(t)=(z(t),p_1(t),q_1(t),\dots,p_k(t),q_k(t)),
\]
be the projection of $\widehat\gamma(t)$ on the reduced thermodynamic phase space.
Note that by (\ref{25jan2806}), the path $\widehat\gamma$ is non-negative in
the extended thermodynamic phase space~$\widehat{\cal T}$.
As $p_e=0$ for $e\in E$, we have
\be
\bal
0\le \hat\lambda\Big(\farc{d\widehat\gamma}{dt}(t)\Big)&=\dot z(t)-S(t)\dot T(t)-
\sum_{j=1}^np_j(t)\dot q_j(t)\\
&=\dot z(t)-S(t)\dot T(t)-\sum_{j=1}^kp_j(t)\dot q_j(t)
-\sum_{j=k+1}^lp_j(t)\dot q_j(t).
\enbal
\ee
We deduce from the aforementioned assumptions on the path $\widehat\gamma(t)$ that
\be\label{25feb1110}
\lambda(\dot\gamma(t))=\dot z(t)-\sum_{j=1}^kp_j(t)\dot q_j(t)\ge S(t)\dot T(t)+\sum_{j=k+1}^lp_j(t)\dot q_j(t)\ge 0.
\ee
Thus, the projection $\gamma(t)$ of $\widehat{\gamma}(t)$ is again non-negative
with respect to the standard contact structure $\lambda$.
We call $\gamma$ {\it a reduced thermodynamic path.}

In the abstract setting of a thermodynamic system where the space of the intensive variables is a manifold, as discussed above Remark~\ref{rem-infty}, admissible
thermodynamic paths are given by non-negative paths $\gamma(t)$, i.e.,
\be\label{25jan3002}
\lambda(\dot{\gamma}(t)) \geq 0\;.
\ee

\begin{exam}\label{exam-magnetic} (The Curie-Weiss magnet)
{\rm
Let $n=k=1$, $p_1 = M$ and $q_1=H$ be the magnetization and the external magnetic field
of a mean-field Ising model (a.k.a. the Curie-Weiss magnet) in the thermodynamic limit.
Here, the temperature $T$ and
the external magnetic field $H$ are the intensive variables, and the entropy $S$ and the magnetization $M$ are the extensive variables.  We assume that the system is subject to the background magnetic field
 $H_{back}$ playing the role of a parameter (see Remark \ref{rem-back} below). The total magnetic field acting on the magnet is then $H+H_{back}$. The equilibrium Legendrian $\widehat{\Lambda}$ in $J^1\R^2$ is given by the equations
(see Section~2 in~\cite{GLP})
\be\label{25jan3102}
\bal
&z=T \ln\left(2\cosh\frac{H+H_{back}+bM}{T}\right) - \frac{b}{2}M^2,\; M = \tanh \left(\frac{H+H_{back}+bM}{T}\right)\;,\\
&S = - \frac{1-M}{2}\ln\frac{1-M}{2} - \frac{1+M}{2}\ln\frac{1+M}{2}\;.
\enbal
\ee
Here, $b >0$ is the spin interaction parameter.  The reduced Legendrian $\Lambda \subset J^1\R$ (where we only reduce the temperature and the entropy)
is given by the first two equations and depends on $T$ and $H_{back}$ as parameters.
We sometimes emphasize this by writing $\Lambda=\Lambda(T,H_{back})$.
Given a thermodynamic temperature-non-decreasing process in $\widehat{\mathcal{T}}=J^1\R^2,$
satisfying (\ref{25jan2806}),
its reduction is a non-negative path in~$\mathcal{T}=J^1\Rm$.}
\end{exam}

\section{Classification of thermodynamic processes}\label{sec:classify}

For our discussion on non-equilibrium thermodynamics, we have to deal with
systems depending on an external parameter 
from  a subset of a Euclidean space.

For instance, we always reduce the entropy and the temperature,
considering the latter to be a parameter. 
As far as the other external parameters $\alpha$ are concerned,  we assume that the Hamiltonian
depends on $\alpha$ in the following way:
$$(q,m,\alpha)\mapsto V'_{\rm int}(m,\alpha)+ q \cdot \overline{V}(m)\;,$$
i.e., that only the internal energy $V'_{\rm int}$ depends on $\alpha$.

Here is an important example. It is convenient to introduce an auxiliary external parameter to describe a jump of an extensive variable.
Take a Hamiltonian $H$ of the form (\ref{25feb604})-(\ref{25feb602}) and write a new Hamiltonian, depending on $\alpha$, as
$$H(q+\alpha,m)= V'_{\rm int}(m,\alpha) + q\cdot\overline{V}(m) = (V_{\rm int}(m)+\alpha \cdot\overline{V} (m))+ q\cdot\overline{V}(m),$$
where $V_{\rm int}(m)$ is the internal energy term in the original $H$ and $V'_{\rm int}(m,\alpha) := V_{\rm int}(m)+\alpha \cdot\overline{V} (m)$ is the new internal energy depending on $\alpha$.
With this notation, the jump $q \to q+c$ can be described as the jump of $\alpha$ from $0$ to
$c$. We refer to Remark \ref{rem-back} for further discussion.

Sometimes, we consider genuine external parameters, such as the exchange parameter of
the Curie-Weiss magnet, see Section \ref{subsec-demag}  below.

We introduce several types of non-equilibrium processes. They differ by their duration - slow (quasi-static), fast, and ultrafast
(instantaneous), and by their ``spread" in the thermodynamic phase space - local and global.

\medskip\noindent
{\sc I. Slow (quasi-static)
processes.} They are caused by a slow change of some parameters of the system.
The change is so slow that, at each moment in time, the system is in equilibrium.
We adopt Assumption \ref{assum-red},
meaning that the decrement of the free energy due to the change of the reduced intensive variables is non-negative. In particular, when the temperature and the entropy are the only reduced variables, we assume that the temperature is non-decreasing \color{black} and, if the temperature is not constant, that the entropy is gauged by an additive constant to be positive. \color{black}  Thus, the slow processes are given by {\it non-negative} families (or paths) of Legendrian (equilibrium) submanifolds. Mathematically, this means that we have an initial Legendrian submanifold $\Lambda \subset \mathcal{T}$ and a map
$$\Gamma: \Lambda \times [0,\tau] \to \mathcal{T},\; \Gamma (x,0) = x \;\;\forall x \in \Lambda,$$
such that for each $t$, the image of $\Gamma(\cdot, t)$ is a Legendrian submanifold
 -- such a map $\Gamma$
is called a (parameterized) Legendrian isotopy, see Section~\ref{sec:slow} below.
In the geometric language, this means that the map $\Gamma$ describing this process is defined on the direct product of the whole equilibrium Legendrian submanifold $\Lambda$ -- and not only of its part -- with the time interval. We assume that each individual path $t\mapsto \Gamma(x,t)$ (for a fixed $x$) is non-negative,
that is, a thermodynamic path  satisfying (\ref{25feb1110}).
Such a Legendrian isotopy is called non-negative, see Section~\ref{sec:slow} below.
A thermodynamic path $t\mapsto \Gamma(x,t)$ appearing in a slow (quasi-static) process $\Gamma$ will be called a {\it slow thermodynamic path}.


We distinguish between two kinds of slow processes. We call a slow process {\it global} if the isotopy $\Gamma$ is a proper embedding (i.e., preimage of every compact subset is compact). \footnote{
As we already stated in the introduction,
by a global process we mean a process
defined at all macroscopic equilibrium states of the system.} An example is provided by the Curie-Weiss magnet whose temperature is slowly increasing, yielding a path of submanifolds  $\Lambda(T,0)$  from Example~\ref{exam-magnetic}. Otherwise we call a slow process {\it local}.

\medskip

Next, we pass to a discussion of fast and ultrafast (instantaneous) processes.
They are triggered by an abrupt change of the external parameters of the system,
say $\alpha_0 \to \alpha_1$. Assume that the initial system was in the equlibrium
with the Gibbs distribution $\rho_{G,\alpha_0}$. The corresponding point in the thermodynamic phase space is $(z^{(0)},p^{(0)},q)$ with $z^{(0)} = -G(q,\rho_{G,\alpha_0},\alpha_0)$ and
\begin{equation} \label{eq-pvsp0}
p^{(0)}= -\int_M \overline{V}(m) \rho_{G,\alpha_0} d\mu\;.
\end{equation}
 What happens to the system
next can be described as follows.

\medskip
\noindent
{\sc Two stage scenario:}

\medskip
\noindent
{\bf Stage 1: Jump.} The macroscopic parameters of the system after the jump
are
$(z^{(1)},p^{(1)},q)$
 with $z^{(1)} = -G(q,\rho_{G,\alpha_0},\alpha_1)$ and
\begin{equation} \label{eq-pvsp1}
p^{(1)}= -\int_M \overline{V}(m) \rho_{G,\alpha_0} d\mu = p^{(0)},
 \end{equation}
cf. \eqref{eq-pvsp0}. It follows that only the free energy changes.
Note that the system remains in the macroscopic state $\rho_{G,\alpha_0}$,
while the microscopic internal energy instantaneously changes from $V_{\rm int}(m,\alpha_0)$ to $V_{\rm int}(m,\alpha_1)$.

\medskip
\noindent
{\bf Stage 2: The Fokker-Planck evolution.} It belongs to the class of fast processes which we are going to discuss it right now.

\medskip\noindent
{\sc II. Fast processes:}  Here we consider a fast\footnote{Here and below 'fast'
means faster than quasi-static and slower than instantaneous. The convergence to the equilibrium is exponential as governed by the first positive eigenvalue of the Fokker-Planck
operator.} evolution of an individual macroscopic state.
(In physical terms, one can think of this evolution as a relaxation of a thermodynamic system in a particular macroscopic state after an abrupt change in the extensive variables -- see e.g. \cite{EPR} for a discussion of such a relaxation of
a Curie-Weiss magnet after a sudden change of its temperature and the exterior magnetic field).
The evolution is  given by a thermodynamic path $\gamma(t) = (z(t), p(t),q)$,  in the reduced
thermodynamic phase space, generated by an evolution of the temperature $T(t)\ge 0$ and of the probability density $\rho(t)\in\cP$.  We consider the temperature~$T(t)$ as a prescribed  non-decreasing  function, with given $T_0=T(0)$ and $T_\infty=\lim_{t\to+\infty}T(t)$.
The intensive variables $q$ remain fixed.
The evolution of $\rho(t)$ is
described by the Fokker-Plank equation
\be\label{25feb2302}
\dot{\rho} = -\nabla_\rho G,
\ee
that involves $T(t)$ as a coefficient, see \cite{EPR}, where such models are elaborated for time-independent temperature.
The corresponding path  $\gamma (t) = (z(t),p(t),q)$, with  $z= - G(T,q,\rho)$,
converges to the equilibrium submanifold corresponding to  the terminal temperature~$T_\infty$ of the process.

We claim that $\gamma(t)$ is necessarily non-negative.
Indeed, recalling \eqref{25feb1002} and since $q$ is kept fixed, we get
\be\label{25feb2304}
\lambda(\dot{\gamma})= \dot{z}= d_\rho G(\nabla_\rho G) +S\dot{T} \geq 0\;,
\ee
and the claim follows.

One can slightly generalize the notion of the fast process by allowing not only the temperature,
but also the external parameter $\alpha(t)$ to move with time.

Let us also emphasize that here we do not assume that at time moment $0$ the system
is in the equilibrium. For instance, at Stage 2 above the evolution starts after the jump, and
hence in general out of equilibrium.  Another example is given by the thermodynamic path corresponding to the transition from the metastable to the stable equilibrium of the Curie-Weiss magnet, see \cite{EPR} for a discussion of metastability and further references.

\medskip\noindent
{\sc III. Ultrafast (instantaneous) processes:} In the thermodynamic limit of certain mean field systems the two stage scenario above may end already at Stage 1. This happens when the limiting
and the terminal Gibbs distributions coincide, so the fast process forming Stage 2  is void: the whole process consists of a jump only.  Assume furthermore that these Gibbs distributions
are given by a delta function on the space of microstates:
\begin{equation} \label{eq-Gibbs-eq}
\rho_{G,\alpha_0} = \rho_{G,\alpha_1} = \delta(m-m_*)\;.
\end{equation}
The latter assumption can be verified for the Curie-Weiss magnet, and for the ideal gas.
It holds true for a more general class of models, essentially due to a version of the large
deviation principle. This will be discussed in a forthcoming paper.

By \eqref{eq-pvsp0} and  \eqref{eq-pvsp1} these assumptions would yield
$p:= p^{(0)}=p^{(1)}= \overline{V}(m^*)$.
Note now that the points $(p,q,z^{(0)})$ and  $(p,q,z^{(1)})$ are endpoints of the Reeb chord
(i.e., a segment parallel to the $z$-axes) connecting the equilibrium Legendrians.

\begin{rem}\label{rem-two-kinds} {\rm The free energy  can change either way in the course of an ultrafast
process. If it decreases,  i.e., $z_1 > z_0$,  there exists a local slow process consisting of germs of Legendrians ``sliding" along the chord. This happens, for instance, with the isochoric piece of the Stirling engine.  It can be realized by a slow local isochoric process in the quasitatic
limit (i.e., when the duration of the cycle tends to infinity), cf. \cite{Stirling-Nature}.
}
\end{rem}

\begin{rem}\label{rem-stab}
{\rm  As explained in \cite{EPR}, if the endpoint of the Reeb chord corresponds to an unstable or metastable equilibrium of the terminal system, there will be further relaxation to the stable equilibrium. We do not address the issue of the stability of equilibria in the context of Reeb chords, and it would be interesting to understand if contact topology can be helpful here.}
\end{rem}

\begin{rem}\label{rem-chord}
{\rm The Reeb chords represent instant jumps from its initial point to the terminal point.
In contrast to the slow and fast processes, the intermediate points of the chord seem to have no physical
meaning.}
\end{rem}

\begin{rem}\label{rem-vsyako}{\rm While ultrafast processes correspond to Reeb chords,
the reverse is not true: Reeb chords can correspond to fast processes and may also appear as paths of individual states in slow
global
processes.
This is illustrated in Example~\ref{example-reeb-gas} below.}
\end{rem}

Let us compare our description of thermodynamic processes with some earlier literature.
Some authors propose to model non-equilibrium processes by contact Hamiltonian flows,
see for instance \cite{BLN,Goto-JMP2015,EP,Goto,Goto24,GLP}. This approach proved successful in several interesting cases, illustrating in particular a thesis from Prigogine’s 1977 Nobel lecture \cite{Prigogine}—that thermodynamic potentials act as Lyapunov functions near equilibrium. At the same time it has limitations in the presence of phase transitions \cite{EP}. In particular, one can prove non-existence (even locally!) of {\it time-independent} contact Hamiltonian that models
some natural relaxation processes \cite{GLP}.
Interestingly enough, positive Legendrian isotopies, proposed as models
of slow thermodynamic processes in the present paper can be realized non-uniquely
by (in general) {\it non-autonomous} ambient contact Hamiltonian flows. This follows from the isotopy extension theorem in
contact topology \cite{Geiges}.
Highlighting the relevance of positive Legendrian isotopies, and in particular
the link with the partial order on the space of Legendrians, is one of the main findings
of the present paper.
For two different contact viewpoints at the Fokker-Planck evolution (cf. the second stage of the scenario above) we refer to \cite{EPR} and \cite{Goto24}. Let us mention also that the GENERIC formalism \cite{GrO,Grmela} provides a simultaneous description of reversible and irreversible thermodynamic processes, including the Fokker-Planck evolution.

The interpretation of the ultrafast processes as Reeb chords between the initial and the terminal Lagrangians appeared first in \cite{EPR} and is further developed in the present paper. One can analyze it through the lens of the correspondence between the quantum mechanics and thermodynamics (we thank N.Nekrasov for this suggestion). Recall that this correspondence
is given by the passage to imaginary time $-i\hbar/T$ which is proportional to the inverse temperature, see e.g. \cite[Section 10-2]{Feynmann-Hibbs} and \cite[formula(2.285)]{Sakurai}. Assume for simplicity that the Lagrangian projections $L_0$ and $L_1$
of the initial and the terminal Legendrians $\Lambda_0$ and $\Lambda_1$, respectively, are graphical of the form
$$L_j = \Big\{p = \frac{\partial f_j}{\partial q}\Big\}\;, \;j=0,1\;, q \in \R^n,$$
and that they intersect transversely in a finite number of points.
On the quantum side $L_j$ defines the space of Lagrangian (or {\it WKB}) states
of the form $a_j(q)e^{i\frac{f_j(q)}{\hbar}}\;,$
with $\int |a_j|^2 dq = 1$, see e.g. \cite{Bates-Weinstein}. The complex-valued functions $a_j$ are assumed to be smooth and compactly supported.
We will denote such a state $\xi_j (a_j,\hbar)$.
The transition probability between $\xi_0 (a_0,\hbar)$ and $\xi_1 (a_1,\hbar)$
is given by
$$p(a_0,a_1, \hbar) := |\langle \xi_0 (a_0,\hbar), \xi_1 (a_1,\hbar)\rangle|^2\;.$$
Here $\langle \cdot,\cdot\rangle$ is the $L_2$-scalar product with respect to the Euclidean
volume. The stationary phase formula (see e.g. \cite[Section 15.6.2]{Guillemin-Sternberg}) yields \begin{equation}\label{eq-stph}
p(a_0,a_1, \hbar) = \hbar^{n/2} \int a_0\overline{a_1} \exp{\left(\frac{i}{\hbar} (f_0-f_1)\right)}d\mu + O(\hbar^{1+n/2})\;,
\end{equation}
where $\mu$ is a complex-valued measure. Crucially, $\mu$  supported in a finite set $\mathcal{Q}$ of the $q$-coordinates of the intersection points from $L_0 \cap L_1$.
Note that (tautologically!) this set is in a one to one correspondence with the set of chords connecting $\Lambda_0$ and $\Lambda_1$. In fact, the points of $\mathcal{Q}$ are exactly the
special values of intensive variables corresponding to ultrafast processes. It would
be interesting to explore this analogy further.

\section{Slow global processes and partial order on Legendrians}\label{sec:slow}

Looking at an equilibrium Legendrian submanifold $L$ of $\mathcal{T}$, the reduced thermodynamic space with the temperature and the entropy being the only reduced variables,
we  cannot, in general, reconstruct the temperature unless we have an additional knowledge about the system. However, the existence of a temperature-nondecreasing slow global process which connects two given properly embedded Legendrians $L_0$ and $L_1$ imposes a contact topological constraint on $L_0$ and $L_1$ which we  explain in this section.

 All Legendrians considered below are assumed to be properly embedded, i.e.
their intersection with every compact subset of $\mathcal{T}$ is compact.

Let us pass for a moment to a more abstract setting of a contact manifold $\Sigma$ with a co-oriented contact structure $\xi = \text{ker}(\alpha)$.
Here, $\alpha$ is a contact form on $\Sigma$ defining the co-orientation.  Let $\Lambda$  be
a properly embedded Legendrian submanifold. A
{\it parameterized
Legendrian isotopy} of $\Lambda$ is a map
$$\Gamma: \Lambda \times [0,1] \to \Sigma,\; \Gamma (x,0) = x \;\;\forall x \in \Lambda,$$
such that each $\Lambda_t: = \Gamma (\Lambda\times t)$ is a Legendrian submanifold.
Let
\[
v_t (x) := \frac{\partial \Gamma}{\partial t} (x,t),~~(x,t)\in \Lambda\times [0,1],
\]
be the vector field of the
isotopy. We say that $\Gamma$ is {\it compactly
supported} if so is the (time-dependent) vector field $v_t$.
We call $\{\Lambda_t\}$, $t\in [0,1]$, a {\it non-parameterized Legendrian isotopy}, or just a {\it Legendrian isotopy}, and
$\Gamma$ its parameterization.
We say that a Legendrian isotopy is compactly supported if it admits a compactly supported parameterization.

Let $[\Lambda]$ be the family of all Legendrians in $(\Sigma,\xi)$
that can be obtained
from a given properly embedded Legendrian $\Lambda$ by a compactly supported Legendrian isotopy. We say that a compactly supported Legendrian isotopy $\Lambda_t$, $t \in [0,1]$, of Legendrians in $[\Lambda]$ is {\it non-negative} if
it admits a parameterization with a vector field $v_t$ satisfying
$\alpha(v_t) \geq 0$ (note that if this holds for one parameterization, then it holds for all the others).
We write~$\Lambda_0 \preceq \Lambda_1$
if there exists a non-negative Legendrian isotopy from $\Lambda_0$ to $\Lambda_1$.

 With the geometric terminology as above, slow global thermodynamic processes are
described by
compactly supported parameterized Legendrian isotopies in the thermodynamic
phase space $\Sigma:=J^1 X$ equipped with the contact
structure as above.
We refer to the slow global processes satisfying Assumption \ref{assum-red} as the
{\it admissible global processes}. Admissible processes correspond to non-negative isotopies.
Recall that when the temperature and the entropy are the only reduced variables, admissibility means that the temperature is non-decreasing.

We are now ready to present a few thermodynamic consequences of contact topology.
 The topological results used below are established in the literature for closed manifolds and are, in part, folklore for open ones.  After stating them, together with their thermodynamic interpretations, we outline their proofs and provide relevant references.

The class $[\Lambda]$ is called {\it orderable} if the binary relation $\preceq$ is a partial order. A Reeb chord between two Legendrians is called {\it non-trivial} if it has positive length (as opposed to a point). The space of intensive variables $X$ below is assumed to be a connected manifold without boundary (either closed, i.e. compact, or open, i.e., non-compact).

\begin{thm}  \label{thm-ord}
For any (open or closed) smooth connected manifold $X$,
the class  $[0_X]$ of the zero section $0_X$ in $J^1 X= T^*X \times \R$
is orderable. If two distinct equilibrium Legendrians $\Lambda_0,\Lambda_1 \in [0_X]$ are related by the partial order,  $\Lambda_0 \preceq \Lambda_1$, there is no
admissible  global process taking~$\Lambda_1$ to $\Lambda_0$.
\end{thm}

\medskip\noindent
The next two results establish a connection between global slow and ultrafast processes.

\begin{thm}\label{thm-B} If two distinct Legendrians $\Lambda_0, \Lambda_1 \in [0_X]$ satisfy
$\Lambda_0 \preceq \Lambda_1$, there exists a non-trivial Reeb chord starting on $\Lambda_0$ and
ending on $\Lambda_1$. In particular, if two equilibrium thermodynamic Legendrians from $[0_X]$ are related by an admissible global process, then
there is  a free-energy non-increasing ultrafast process which starts at an equilibrium state of the first system and ends at an equilibrium state of the second system.
\end{thm}

\begin{thm}\label{thm-C}
For every two distinct Legendrians in $[0_X]$ there exists a non-trivial Reeb chord
connecting them (in an unspecified order). If these are the equilibrium Legendrians
of two thermodynamic systems, there exists an ultrafast process that starts at
an equilibrium state of one system and ends at an equilibrium state of the other.
\end{thm}

\medskip
\noindent
{\bf Comments on the proofs:}

\noindent
{\sc Case I:} $X$ is a closed manifold.

In this case, Theorem \ref{thm-ord} follows from \cite[Proposition 5]{CN}  and \cite[Theorem 1]{CFP}.
Theorem \ref{thm-B} can be deduced from the theory of generating functions
and their spectral selectors \cite{CN1,CFP,Z} along the following lines. Assume without loss
of generality that $\Lambda_0$ is the zero section.  Let~$c_+$ and $c_-$ be the spectral selectors
corresponding to the fundamental class (i.e., the class of $X$)  and the class of the point in the homology $H_*(X,\mathbb{Z}_2)$ with the  $\mathbb{Z}_2$-coefficients, respectively.  If $\Lambda_0 \preceq \Lambda_1$, then
$c_+(\Lambda_1) \geq 0$. Since~$\Lambda_0 \neq \Lambda_1$, we have $c_+(\Lambda_1) >0$.
 It follows from the spectrality property of the spectral selectors (see e.g. Proposition 2.4(i) in \cite{Z}) that there exists a chord of time-length $c_+(\lambda_1)$ starting on $\Lambda_0$ and ending on~$\Lambda_1$, as required.  Theorem \ref{thm-C} follows from the fact that since $\Lambda_0 \neq \Lambda_1$, at least one of~$c_-(\Lambda_1)$ or $c_+(\Lambda_1)$ is distinct from~$0$.

A different argument which deduces Theorems \ref{thm-B} and \ref{thm-C} from the orderability of the class $[0_X]$ is discovered in a recent paper by Allais and Arlove \cite{AA}, see Corollary 1.5 and Theorem 4.7 therein. Interestingly enough, this argument works on
arbitrary contact manifolds, with $[0_X]$ being replaced by arbitrary orderable Legendrian isotopy class of a closed Legendrian submanifold. Our exposition is influenced by this paper.

\medskip
\noindent
{\sc Case II:} $X$ is an open manifold.

When $X = \R^n$, all the results follow from the spectral selectors constructed in \cite{B}.
Existence of generating functions and the first properties of spectral selectors on
general open manifold is elaborated in \cite{CN1}. A quick reduction of the open case
to the closed one is obtained by the following construction. Choose an  exhausting proper bounded
from below function $f: X \to \R$. Set $X_C := \{f \leq C\}$, where $C$ is a regular value
of $f$. Let $X_C^\circ$ be the interior of $X_C$.
Any compactly supported Legendrian isotopy in $J^1X$ is supported in $J^1X_C^\circ$ for some
large $C$. Now do the following trick: pass to the double of $X_C$, denoted by $Y_C$ --
in other words, glue $X_C$ with itself along the boundary $\{f=C\}$.
Crucially, $Y_C$ is a closed manifold.

Theorem \ref{thm-ord} follows from the fact that every non-negative Legendrian loop in $J^1X$
corresponds to such a loop in $J^1Y_C$ and hence is constant by Theorem \ref{thm-ord} applied to $Y_C$.
Theorems~\ref{thm-B} and \ref{thm-C} follow from the fact that in the closed case the chords constructed
in  Step 1 are non-trivial. Thus, when  we apply these theorems to $J^1Y_C$, we get the chords lying in
$J^1X_C$. This completes the proof.

\qed.

\begin{rem} \label{rem-A} {\rm We have considered two models of thermodynamic processes:
thermodynamic
paths and paths of Legendrian submanifolds, both in the thermodynamic phase space. It is
natural to consider a model which interpolates
between these two, as follows.  Given an initial Legendrian submanifold $\Lambda_0 \subset \mathcal{T}$,
we consider a smooth
 (not necessarily Legendrian!)
isotopy $$\Gamma: \Lambda_0 \times [0,1] \to \mathcal{T},\; \Gamma (x,0) = x \;\;\forall x \in \Lambda,$$
such that $\Lambda_1:= \Gamma(\Lambda_0\times 1)$ is a also Legendrian submanifold
(but intermediate $\Lambda_t : = \Gamma (\Lambda\times t)$ do not have to be Legendrian).
Additionally, we assume that all individual paths $t\mapsto \Gamma(x,t)$ are non-negative (or, as a variation of this problem, positive, i.e., transversal to the contact hyperplanes and respecting the coorientation). We call such smooth isotopies non-negative/positive.
In other words, we model a process by a path of submanifolds
connecting the initial and the terminal equilibrium submanifolds without assuming that in
the course of the process the system is  always in an equilibrium,
but insisting that the individual
trajectory of every state agrees with the laws of thermodynamics,  as in (\ref{25jan2806})
and~(\ref{25feb1110}). At the moment, it is not
clear to which extent the non-negativity/positivity assumption is restrictive in this context.
 On one hand, there do exist examples of pairs of
(compact) Legendrian submanifolds that can be connected by a non-negative smooth isotopy
but cannot be connected by a non-negative Legendrian isotopy.
For instance,  it is not hard to construct a non-negative smooth isotopy between
$\{z=1, p=0\}$ and $\{z=-1, p=0\}$ in $J^1S^1$, though
there is no non-negative Legendrian isotopy between these curves.
On the other hand, it is unclear how to characterize or detect
such pairs in general.
}
\end{rem}

\begin{rem}\label{rem-LY} {\rm We refer to \cite{LY} for appearance of some other order relations in thermodynamics. It would be interesting to relate the discussion in \cite{LY} with
the partial order on Legendrians considered above.} \end{rem}

\section{Ultrafast processes as Reeb chords}\label{sec-ultra}

In this section, we discuss two examples of ultrafast processes: an ideal gas and the Curie-Weiss magnet.
The ideal gas model is directly related to the Stirling engine, a well known device, where the
ultrafast process can be observed experimentally~\cite{Stirling-Nature}.  In the Curie-Weiss magnet context,
we also establish the existence of an ultrafast process connecting an equilibrium of a magnet at a given temperature
to a perturbation of the magnet at a larger temperature.

\subsection{An isochoric temperature and pressure jump}\label{sec:ideal-gas}

Consider the lattice gas consisting of $N$ particles occupying $V$ sites at temperature $T$ in the thermodynamic limit $V,N \to +\infty$ so that $V/N \to v$. Furthermore, we assume that the gas
is diluted, $v \gg 1$.
 This is a model approximating the ideal gas -- see \cite[Section 1.3.2]{FV} for details.

In this model, the temperature $T$
and the pressure  $P$  are the intensive variables, while the volume $v$ is the extensive variable.
We will not need to use the entropy in this example.
If the Bolzmann and the gas constants
are taken to be $1$, in an equilibrium we have
\be\label{25feb1116}
(P+P_{\rm back})v=T.
\ee
Here, $P_{\rm back}$ is a background pressure that we treat as a physical parameter of the system.
If  the molecules do not interact so that the internal energy vanishes, then
the Hamiltonian in this model is linear, of the form (\ref{25feb604})-(\ref{25feb602}) with $V_{\rm int}(m)=0$. Then,
a straightforward calculation following~\cite{FV} shows that  the equilibrium free energy~$G$
in the ideal gas approximation is
\[
G= -T\ln v = T\ln((P+P_{\rm back})/T).
\]
We warn the reader that in this model the internal energy per particle vanishes, as opposed to the standard model of the ideal gas where it equals $\frac{3}{2}T/v$.

Assume now that the temperature jumped from $T_0$ to $T_1 > T_0$ and
the background pressure
 simultaneously
jumped from $P_{\rm back}=0$ to  $P_{\rm back}=c >0$,
while the volume is kept fixed. One can think that the gas
is contained in a chamber with one of the walls being a piston of area $a$,
and one instantly changes the temperature
and simultaneously switches on a force~$f$ applied
to the piston and directed towards the gas so that $c = f/a$.

Given the parameters $T_0$, $T_1$ and $c$, there exist unique values of the volume $v$
and the initial pressure $P_0$ when such a jump corresponds to the ultrafast process.  This happens if, after the jump, the system lands in the equilibrium of the terminal system, so that
\begin{equation}\label{eq-1-v}
v= \farc{T_0}{P_0} = \farc{T_1}{P_0 +c}\;.
\end{equation}
We compute that
\be\label{25feb1114}
P_0 = \frac{cT_0}{T_1-T_0}, \; v= \frac{T_1-T_0}{c}\;.
\ee
In particular, if the background pressure jump   $c \ll 1$ is small, and $T_1-T_0=O(1)$,
the volume~$v$ is large in accordance to the ideal gas approximation.
 Recall that this ultrafast process corresponds to the Reeb chord
connecting initial and terminal equilibrium Legendrians. In the $(v,P)$-plain this chord corresponds to the {\it intersection point} of the Lagrangian projections of these
Legendrians. This is visualized on Figure~\ref{fig-jump-gas}.
\begin{figure}[htb]
\includegraphics[width=0.6\textwidth]{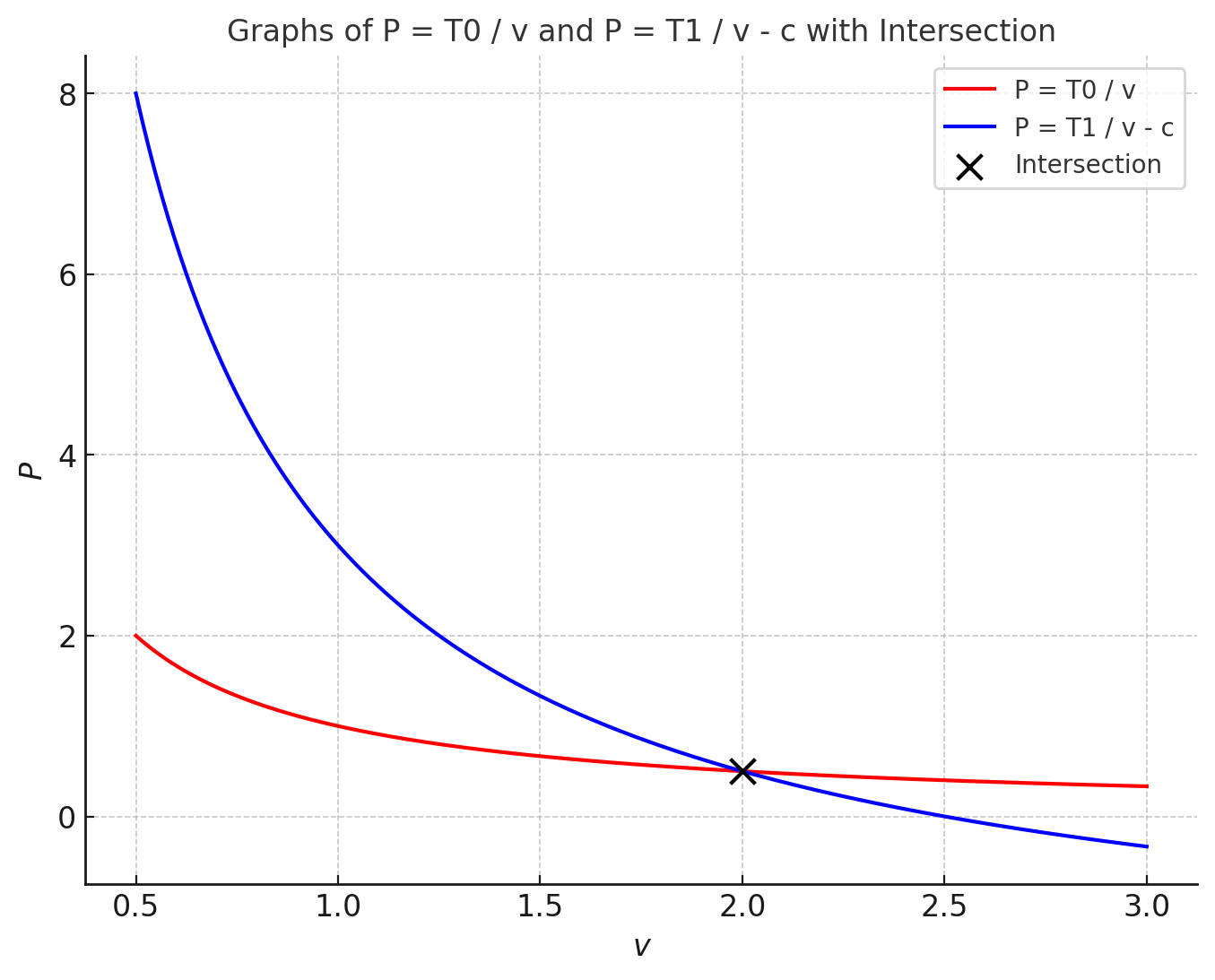}
\caption{ Ultrafast process for $T_0=1,T_1=5,c=2$}\label{fig-jump-gas}
\end{figure}
Note that the change in the Gibbs free energy is negative if the pressure jump $c<T_1-T_0$:
\begin{equation}\label{eq-G}
G_1 -G_0 = (T_0-T_1) \ln v < 0\;.
\end{equation}
Let us mention that isochoric temperature and pressure jumps appear as two of the four segments of
the Stirling engine cycle\footnote{We thank S. Goto for the reference to the Stirling engine.
For a concise description of the Stirling engine see https://en.wikipedia.org/wiki/Stirling\_engine. The attribution for Figure~\ref{fig-stirling}
is  to Cristian Quinzacara, CC BY-SA 4.0 https://creativecommons.org/licenses/by-sa/4.0, via Wikimedia Commons.}
 as depicted in Figure~\ref{fig-stirling}. It is tempting to view the  segment of the engine
where the temperature goes up at a fixed volume as  precisely the above Reed chord,
since in its derivation we have only used the ideal gas equilibrium law~$PV=T$ that also holds for the Stirling engine.  We shall see below that this is indeed the case, albeit
after the shift of the terminal Legendrian submanifold along the $P$-axis by a
suitable background pressure.
Let us also mention that a recent paper~\cite{Stirling-Nature}
 dealing with a microscale Stirling engine
 states that
``the isochoric transitions were nearly
instantaneous and occurred on millisecond time scales".
\begin{figure}[htb]
\includegraphics[width=0.6\textwidth]{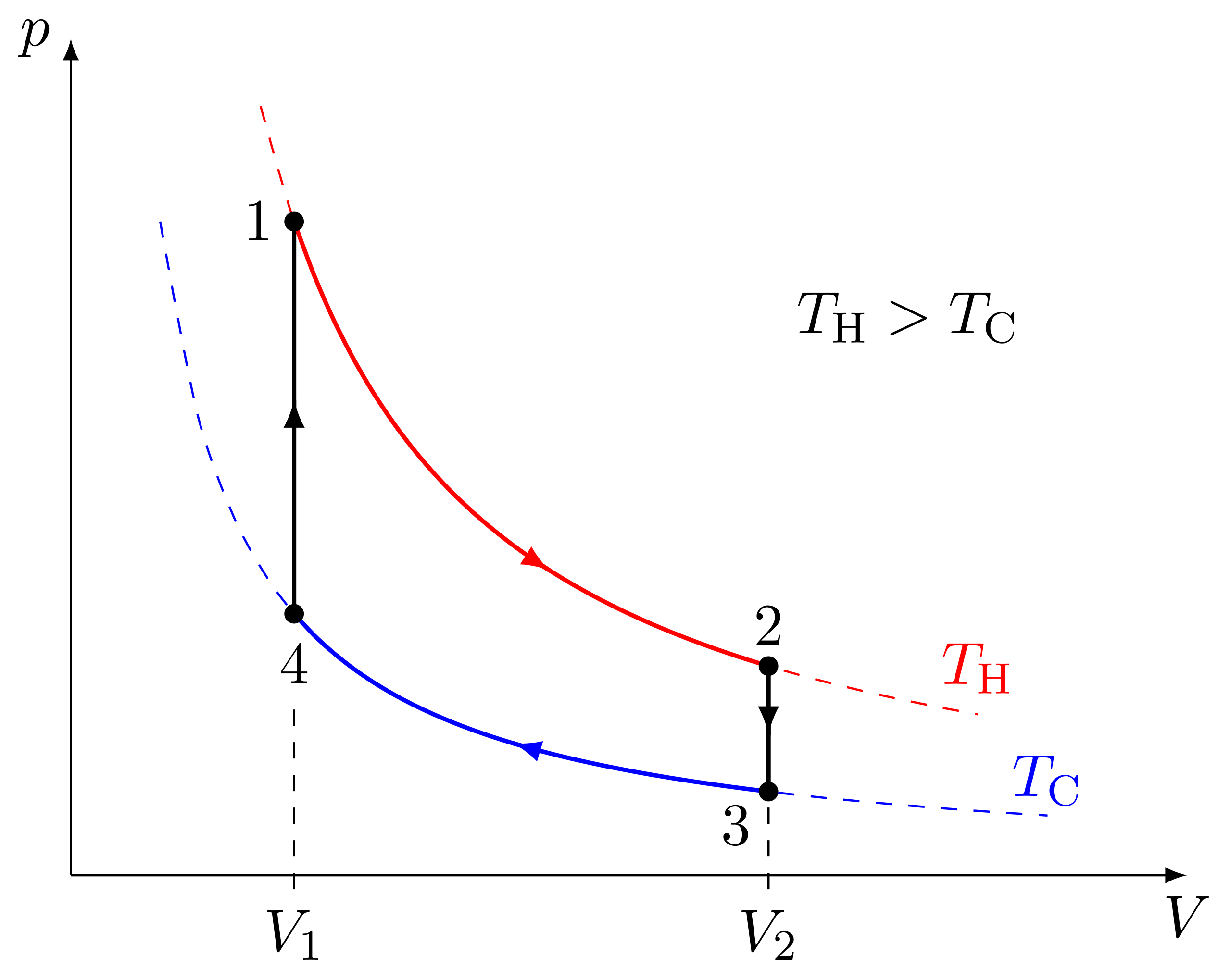}
\caption{The Stirling engine cycle, with $T_0=T_C<T_H=T_1$.}\label{fig-stirling}
\end{figure}

\begin{rem}\label{rem-Maxwell}
{\rm Our model ignores the Maxwell–Boltzmann distribution of particle velocities. A possible objection to this simplification is that the Maxwell–Boltzmann distribution depends on temperature, so a temperature jump would inevitably drive the system out of equilibrium. To address this, we tacitly assume that relaxation of the velocities to equilibrium is much faster than that of
the coordinates. The physical feasibility of this assumption will be discussed elsewhere.}
\end{rem}

\begin{rem}
\label{rem-back}
{\rm When discussing an instantaneous isochoric process, as in the Stirling engine, temperature acts as a parameter, while pressure is a macroscopic variable. Introducing an artificial parameter, the background pressure, such that the total pressure is given by
$P+P_{\rm back}$, allows us to describe the
 simultaneous
jump of two parameters:
$T$ jumps from $T_0$ to $T_1$, and $P_{\rm back}$ from $0$ to $c$.
More generally, this trick applies to processes in which intensive variables exhibit instantaneous jumps while the corresponding extensive variables remain unchanged,
see Section \ref{subsec-mag} for another example.
 This not only
streamlines the exposition, but, more importantly, gives rise to the connection between such processes and Reeb chords.}
\end{rem}

Let us explain now how such ultrafast process corresponds to a Reeb chord
between two Legendrian submanifolds. We work in the reduced thermodynamic phase space $(z,p,q)$
where~$z=-G, p=v, q=-P$. This notation is in accordance with the fact that  volume~$v$ is an extensive variable and the pressure $P$
is the intensive variable\footnote{We hope that the difference between $p$ and $P$ will not cause confusion for the reader.}.

The Legendrian submanifold corresponding to the ideal gas is given by
\be\label{25jan3108}
\Lambda(T,P_{\rm back})= \{(z,p,q) \in \R^3\;:\; z= \phi_T(q-P_{\rm back}),\; p = \phi'_T(q-P_{\rm back})\}\;,
\ee
where $\phi_T(q) = -T\ln(-q/T)$. Let us set
$\Lambda_0 = \Lambda(T_0,0)$ and $\Lambda_1 = \Lambda(T_1,c)$.

To detect the chord connecting $\Lambda_0$ and $\Lambda_1$, one needs to solve the equation
$$p_0 = p_1 = -T_0/q = -T_1/(q -  c),$$
which coincides with \eqref{eq-1-v}. Thus, the  ultrafast process discussed above,
corresponds exactly to the Reeb chord connecting $\Lambda_0$ with $\Lambda_1$.
Furthermore, the orientation of the chord is given by
$\partial/\partial z$ since $z_1 - z_0 >0$ by \eqref{eq-G} and, as we recall, $z=-G$.

Let us illustrate the Reeb chord after the change of variables
\[
\overline{z} = z- \phi_{T_{0}}(q), ~~ \overline{p} = p -  \phi'_{T_{0}}(q), ~~\overline{q} = q\;.
\]
It preserves the Gibbs form, sends $\Lambda_0$ to the zero section $\overline\Lambda_0$, and $\Lambda_1$ to
\[
\overline\Lambda_1=\{(z,p,q) \in \R^3\;:~\overline{z} = \psi (\overline{q}),~
\overline{p}= \psi' (\overline{q})\},\;
\]
where
$$\psi(\overline{q})= \phi_{T_1}(\overline{q}-c) -\phi_{T_0}(\overline{q})\;.$$
The projections of the resulting Legendrians and the chord onto the $(q,z)$ - plane (the front projection)
are given on Fig.~3.

\begin{figure}[htb]
\includegraphics[width=0.6\textwidth]{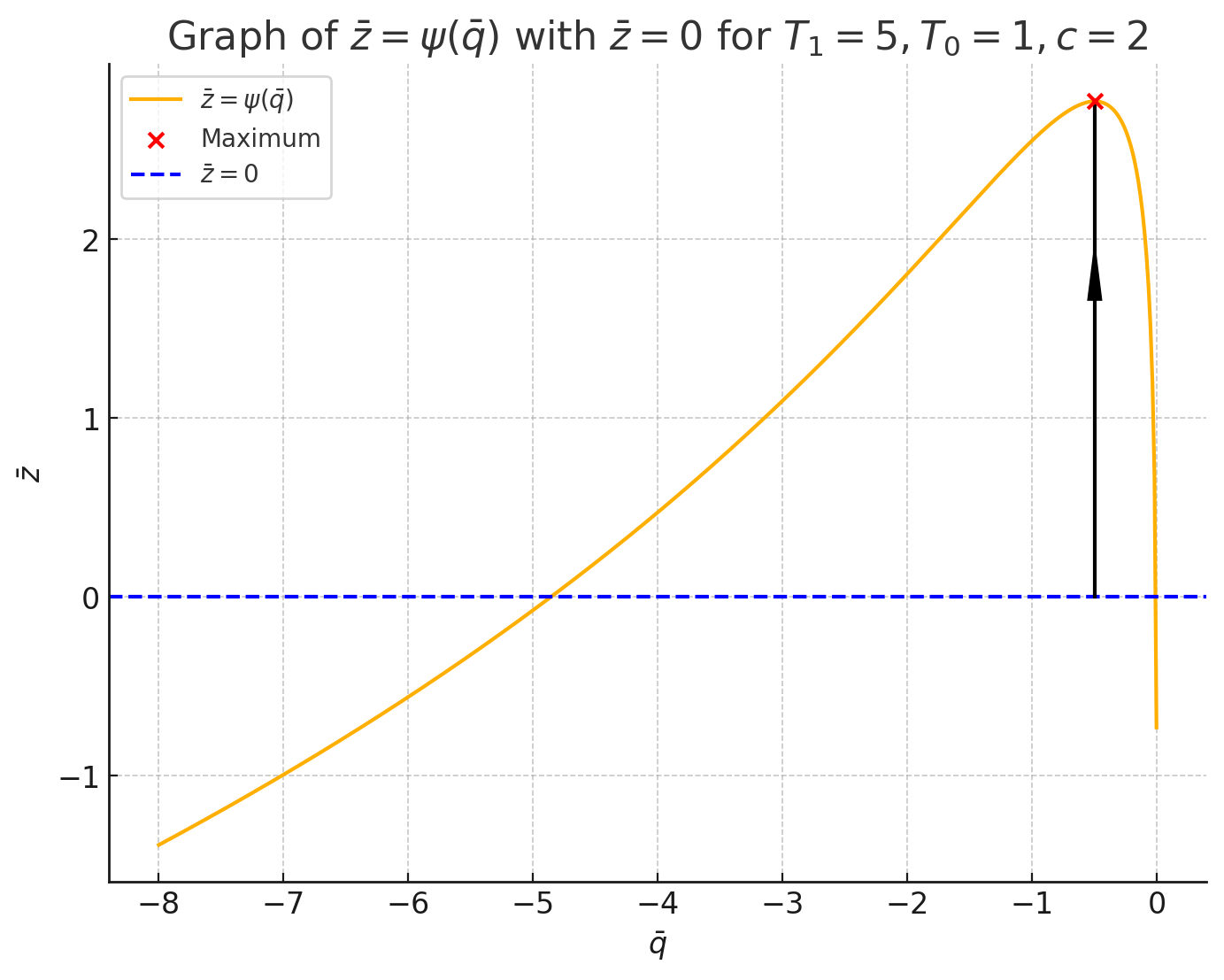}
\caption{ Reeb chord for $T_0=1,T_1=5,c=2$}
\end{figure}

\begin{rem}\label{rem-chord-1}
{\rm It sounds likely that there is no non-negative Legendrian isotopy
(that is, a slow global
thermodynamic process) connecting $\Lambda_0$ and $\Lambda_1$.
To make this statement rigorous, one has to find a way to deal with the lack of compactness in this example, cf. Remark \ref{rem-infty}
above.
For instance, let us modify $\overline\Lambda_1$ below $\overline\Lambda_0$ so that it coincides with the line $\{z=-1, p=0\}$ outside
a compact set. Twist now $J^1\R$ into $J^1S^1$ by taking the quotient by
$(z,p,q) \mapsto (z,p, q+C)$ with a large constant $C$, so that $\overline\Lambda_0$ and $\overline\Lambda_1$ are modified to the
Legendrian circles $\Lambda'_0$ and $\Lambda'_1$. One can show that these circles
are {\it incomparable} with respect to the partial order. This can be proved with the help
of Legendrian spectral invariants which are monotone with respect to non-negative Legendrian isotopies, see \cite[Section 4.1]{CFP}.
 }
\end{rem}

\begin{exam}\label{example-reeb-gas}{\rm
Here, we illustrate that Reeb chords can also
correspond to a
slow thermodynamic path, still on the ideal gas example.
Suppose the temperature increases slowly with time \( 0\le t \leq 1 \), given by \( T(t) = T_0 + (T_1-T_0)t \),
and the external pressure also increases slowly \( P_{\rm ext}(t) = c t \), while the volume remains constant,~\( v(t) = v \).
%
The Lagrangian projection of the equilibrium Legendrian at time \( t \) onto the \( (p, q) \)-plane is given by
\[
p(-q +c t) = T_0 + (T_1-T_0)t.
\]
It follows that each of this family of
slow thermodynamic
  paths
 passes through the point given by (\ref{25feb1114}):
\[
 (p_*, q_*) =\Big(\farc{T_1-T_0}{c}, -\farc{cT_0}{T_1-T_0}\Big).
 \]
In addition, the negative free energy increases during this process:
\[
z(t) = -G(t)= T(t) \ln v(t) =  \bigl(T + (T_1-T_0)t\bigr) \ln v, \quad v \gg 1.
\]
Thus, the individual trajectory of the macroscopic state \( (z_0, p_*, q_*) \) is a Reeb chord.  Note that the
 Legendrian
isotopy described above is non-negative, and hence represents a legitimate thermodynamic process,  only if we consider a small neighborhood
of the point $(T_0\ln q_*, p_*, q_*)$ on the initial Legendrian.}
\end{exam}

\subsection{A temperature and magnetic field jump with constant magnetization}\label{subsec-mag}

We start with  a discussion based on \cite{EPR}.
Consider the Curie-Weiss magnet in the thermodynamic limit, see Example \ref{exam-magnetic}. As the external magnetic field
$H$ is an intensive variable, and the magnetization $M$ is an extensive one, we use the notation
$p=M\in\Rm$ and~$q=H\in\Rm$.  Recall (cf. also Remark \ref{rem-back}) that the total magnetic field is given by $H+H_{\rm back}$, where $H_{\rm back}$ is the background magnetic field, playing
a role of the parameter. Suppose that the temperature of the Curie-Weiss magnet jumped from $T_0$ to $T_1 > T_0$, and that the background magnetic field jumped from $0$ to $c \in \R$. The reduced Legendrian equilibrium submanifold $\Lambda(T,H_{\rm back})$ is defined by
the first two equations in (\ref{25jan3102}).
Let us set $\Lambda_0= \Lambda(T_0,0)$, and let $\Lambda_1 = \Lambda(T_1,c)$.

We claim that there is a unique chord Reeb chord joining $\Lambda_0$ and $\Lambda_1$.
Indeed, it follows from the second equation in (\ref{25jan3102}) that such a chord projects to the point $(p,q)$ if and only if
\be\label{25jan3104}
q+c +bp = T_1 \tanh^{-1}p, ~~ q +bp = T_0 \tanh^{-1}p\;,
\ee
which gives
\[
p = \tanh \frac{c}{T_1-T_0}\;.
\]
Substituting back into either of the two equations in (\ref{25jan3104}), we get
\begin{equation}\label{eq-q}
q= \frac{cT_0}{T_1-T_0} - b\tanh \frac{c}{T_1-T_0}\;.
\end{equation}
Moreover,  the free energy
\[
z=T \ln\Big(2\cosh\frac{q+bp}{T}\Big) - \frac{b}{2}p^2
\]
is an increasing function of $T$ for  $p$ and  $q$ fixed because
\be
\bal
\pdr{z}{T}&= \ln 2+ \ln\Big(\cosh\frac{A}{T}\Big)+\farc{T}{\cosh(A/T)}\sinh\Big(\farc{A}{T}\Big)\Big(-\farc{A}{T^2}\Big)\\
&=
\ln2 +\ln\cosh(A/T)-\farc{A}{T}\tanh\Big(\farc{A}{T}\Big)=\ln\Big(\farc{e^{A/T}+e^{-A/T}}{e^{(A/T)\tanh(A/T)}}\Big)>0.
\enbal
\ee
Here, we have set $A=b+pq$.
Thus, the
chord goes in the direction of the Reeb field~$\partial/\partial z$.
The claim follows.

As discussed in \cite{EPR}, the chords correspond to ultrafast relaxation processes
{\it in the thermodynamic limit} under an additional assumption that they connect
stable (as opposed to unstable or metastable) equilibria of the Curie-Weiss model, (see Remark \ref{rem-stab} above). This happens, for instance, if $bT_1 < 1$,  or if $c>0$ and for $q$ given by \eqref{eq-q} both $q$ and $q+c$ have the same sign.

Keeping the previous notation, we have the following result.

\begin{thm}\label{thm-main}
Let $g_t$, $t \in [0,1]$ be
a compactly supported contact isotopy of $J^1\R$ such that
\be\label{25jan3114}
g_t(\Lambda_1) \cap \Lambda_0 = \emptyset \;\;\;\forall t \in [0,1]\;.
\ee
Then there exists the Reeb chord starting at $\Lambda_0$ and ending at $g_1(\Lambda_1)$, i.e. an instantaneous
relaxation process corresponding to the above jump.
\end{thm}

\begin{proof}  The Legendrian $\Lambda(T,H_{\rm back})$ can be equivalently represented as
\be\label{25jan3112}
z=\phi_T(q+H_{\rm back}+bp)-\farc{bp^2}{2},~~p=\phi_T'(q+H_{\rm back}+bp).
\ee
 Here, we have set
\be
\phi_T(q)=T\ln\Big(2\cosh \Big(\farc{q}{T}\Big)\Big).
\ee
We make a change of variables as in \cite{EP}
\begin{equation}\label{eq-change-var}
Q= q+bp, \; P= p-\phi'_{T_0}(Q),\; Z= z- \phi_{T_0}(Q) + \farc{bp^2}{2}\;,
\end{equation}
which preserves the Gibbs form:
\be
\bal
dZ-PdQ&=dz-\phi_{T_0}'(Q)dq-b\phi_{T_0}'(Q)dp+bpdp-(p-\phi'_{T_0}(Q))(dq+bdp)\\
&=dz+(-\phi_{T_0}'(Q)-p+\phi_{T_0}'(Q))dq+(-b\phi_{T_0}'(Q)+bp-bp+\phi'_{T_0}(Q))dp\\
&= dz-pdq.
\enbal
\ee
Hence, it also preserves the Reeb chords. It follows from (\ref{25jan3112}) and (\ref{eq-change-var}) that under this change
of variables, the Legendrian submanifold    $\Lambda_0=\Lambda(T_0,0)$ becomes
the zero section
\[
L=\{Z=0,~P=0\}.
\]
The Legendrian  $\Lambda_1=\Lambda(T_1,c)$  after the change of variables
becomes
\begin{equation}\label{eq-newLag}
K:= \{Z =\psi(Q):=  \phi_{T_1}(Q+c) - \phi_{T_0}(Q),\; P = \psi'(Q)\}\;.
\end{equation}
An elementary analysis shows that $\psi$ is asymptotic to the lines  $
\{z= \pm c\}$ as $Q \to \pm \infty$,
and has its unique maximum at $Q_*= cT_0/(T_1 - T_0)$, compare with the previous section and see Fig.~4.
This maximum corresponds to a Reeb chord from the zero section $L$  to $K$.

\begin{figure}[htb]
\includegraphics[width=0.6\textwidth]{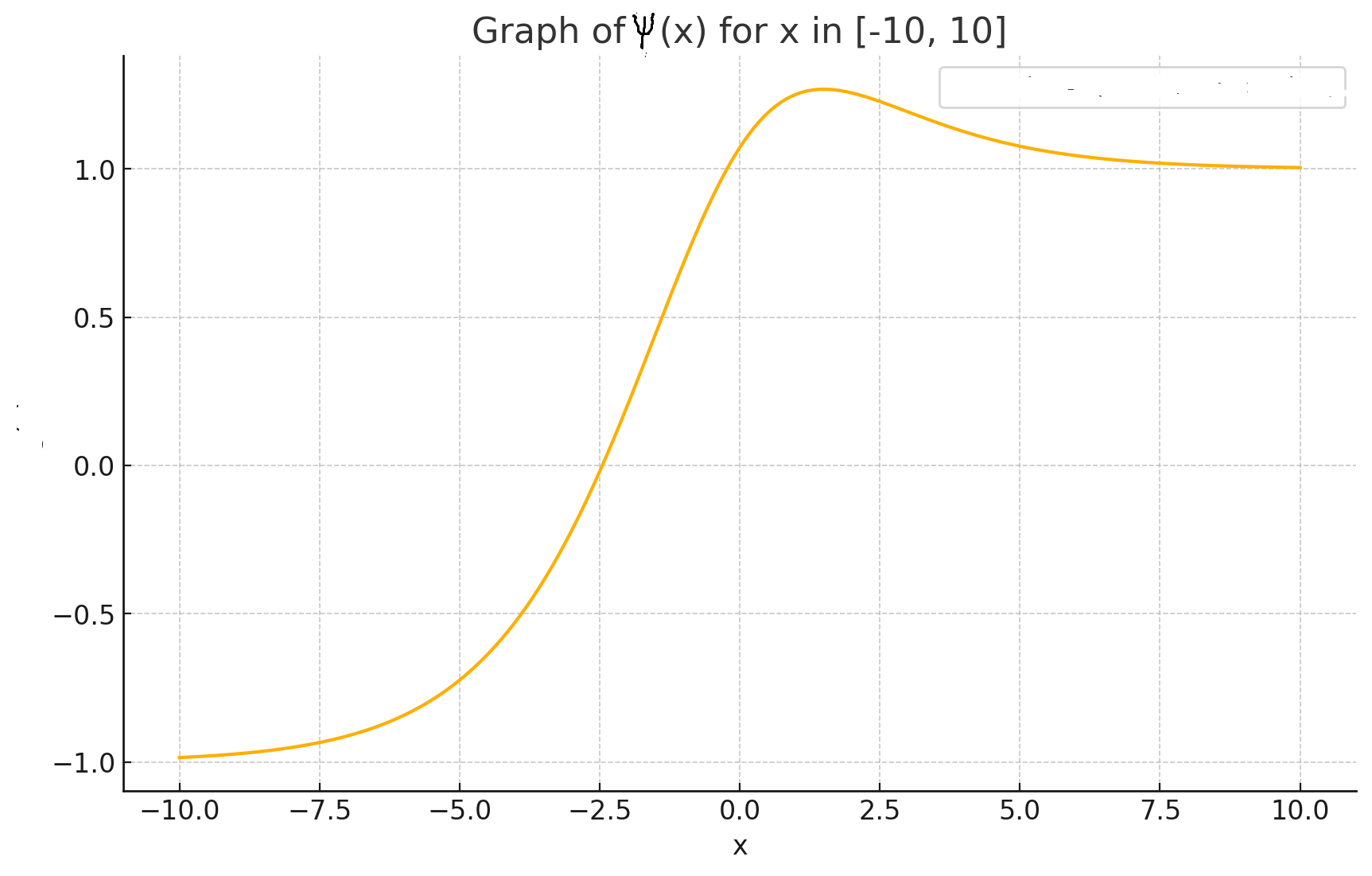}
\caption{$b=1,\beta_0 = 0.5, \beta_1 = 0.3, c= 1$}
\end{figure}

Let now $g_t$, $t \in [0,1]$ be
a compactly supported contact isotopy of $J^1\R$ satisfying (\ref{25jan3114}).
Fix sufficiently large positive numbers $R > r >0$,
let $A_R$ be the cube $[-R,R]^3 \subset \R^3 = J^1\R$ and denote
by $S_R$ the circle $\R/ R\Z$ in the $Q$-variable.
For a Legendrian submanifold $Y$ which coincides with $K$ outside
$A_r$,  we denote by $\widehat{Y}$ the Legendrian submanifold of $J^1S_R$ obtained by the following procedure  (see Fig.~5).
First, interpolate   between $Y$ and the Legendrian submanifold
\[
L_c=\{z=-c, p=0\}
\]
inside
the interior of $A_R$ without creating new chords from the zero section
$L$ to the resulting  submanifold $\widehat{Y}$. Then,
glue together the points of this submanifold
corresponding to $Q=-R$ and $Q=R$. This is possible simply because   $L_c$ does not depend on $Q$.
\begin{figure}[htb]
\includegraphics[width=0.8\textwidth]{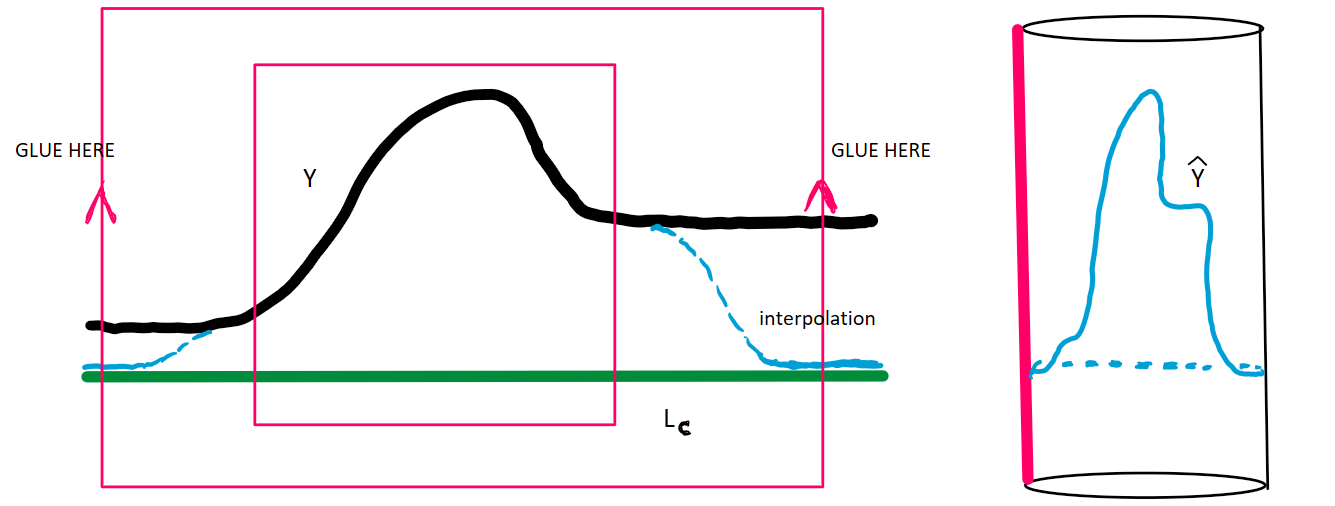}
\caption{Interpolation and gluing in $(Z,Q)$-plane}
\end{figure}
Observe that there is
a unique non-degenerate chord from $\widehat{L}$ to $\widehat{K}$
corresponding to $Q_*$, and that these submanifolds are explicitly Legendrian isotopic
by changing $T$ from $T_0$ to $T_1$. Thus, by a result from~\cite{EP-LCH},  the submaniflds~$\widehat{L}$ to $\widehat{K}$  are interlinked
(see \cite{EP-LCH,EP} for the definition of interlinking).
In particular, if $g_t$, $t \in [0,1]$, is a contact isotopy of $\R^3$ supported in $Q_r$ such that $g_t(K)$ does not touch $L$ for
all $t$, it descends to $J^1S_R$, and hence $\widehat{g_1(K)}$ and $\widehat{L}$ are interlinked.
Since the interpolation above does not create new chords, we get a chord between $g_1(K)$ and $L$. Returning to the original coordinates $(z,p,q)$ we get the statement of the theorem.
\end{proof}




\subsection{An ultrafast (de)magnetization for the Curie-Weiss model}\label{subsec-demag}

Consider the Curie-Weiss magnet at temperature $T$. Assume that the background magnetic field
vanishes. Put $\psi(x) = T\ln \left(2 \cosh (x/T)\right)$.
Recall from Example \ref{exam-magnetic} that in the equilibrium
\begin{equation} \label{eq-z-CWe}
z= \psi(q+bp) - \frac{bp^2}{2}\;,
\end{equation}
\begin{equation} \label{eq-p-CWe}
p = \psi'(q+bp)\;,
\end{equation}
and
\begin{equation} \label{eq-q-CWe}
q= -bp + Tu(p)\;,
\end{equation}
with
$$u(p) = \frac{1}{2} \ln \frac{1+p}{1-p}\;.$$
We view $z$ and $p$ as functions of $b$, with $T$ and $q$ being fixed.
Differentiating \eqref{eq-z-CWe} by $b$ and substituting \eqref{eq-p-CWe} we get that
$$\frac{dz}{db} = p^2/2\;.$$

Furthermore, by \eqref{eq-q-CWe}
$$T^{-1}b = \frac{u(p) - T^{-1}q}{p}\;.$$
Differentiating by $p$ we get
$$T^{-1} \frac{db}{dp} = \frac{\left(pu'(p) - u(p)\right) + T^{-1}q}{p^2}\;.$$
The expression in brackets in the nominator is the $y$-intercept of the function $u$
taken with the opposite sign. Since for $p > 0$ the function $u$ is strictly convex with $u(0)=0$, its intercept is negative. It follows that
for $p >0$, $dz/db > 0$ and $db/dp > 0$ (and hence $dp/db >0$).

In particular, if as the result of certain ultrafast process, the exchange parameter (called sometimes {\it coupling}) $b$ is decreasing, the free energy increases and magnetization drops; and if $b$ is increasing, the free energy decreases and the magnetization is increasing. In some magnetic systems the manipulation of the parameter $b$ can be achieved by
ultrashort laser pulses applied to the system ( see \cite{MPE, M-nature};
we note that none of these papers use mean-field models; we introduce them here as an approximation.)

\noindent{\bf Data availability statement.} This paper does not use any data.

\noindent{\bf Conflict of interest statement.} On behalf of all authors, the corresponding author states that there is no conflict of interest.

\end{document}